\documentclass[reqno]{amsart}
\usepackage{amssymb, amsmath}
\newtheorem{theorem}{Theorem}
\newtheorem{acknowledgement}[theorem]{Acknowledgement}

\newtheorem{corollary}[theorem]{Corollary}

\newtheorem{lemma}[theorem]{Lemma}

\newtheorem{proposition}[theorem]{Proposition}

\begin{document}

\title[The Solvability of Inverse Problem]
{Necessary and Sufficient Conditions for the Solvability of Inverse Problem for a Class of
Dirac Operators}
\author[Kh. R. Mamedov and O. Akcay]{Kh. R. Mamedov$^{1,*}$ and O. Akcay$^{2}$} 
\address{$^{1,2}$Science and Letters Faculty, Mathematics Department,
Mersin University, 33343, Turkey}
\thanks{$^{*}$corresponding author}
\email{$^{1}$hanlar@mersin.edu.tr and $^{2}$ozge.akcy@gmail.com}
\date{}
\maketitle
\begin{abstract}
In this paper, we consider a problem for the first order 
Dirac differential equations system with spectral
parameter dependent in boundary condition. The asymptotic behaviors of
eigenvalues, eigenfunctions and normalizing numbers of this system are investigated. 
The expansion formula with respect to eigenfunctions is obtained and
Parseval equality is given. The main theorem on necessary and sufficient conditions for the solvabilty of 
inverse problem is proved and the algorithm of reconstruction of potential from spectral data (the sets of eigenvalues
and normalizing numbers) is given.
\end{abstract}

\section{Introduction}

We consider the boundary value problem generated by one dimensional Dirac differential equations system
\begin{equation}\label{1}
By'+\Omega(x)y=\lambda y, \ \ \ 0<x<\pi
\end{equation}
with boundary condition
\begin{equation}\label{2}
\begin{array}{c}
U(y):= y_{1}(0)=0, \\
V(y):= \left(\lambda+h_{1}\right)y_{1}(\pi)+h_{2}y_{2}(\pi)=0,
\end{array}
\end{equation}
where
\[
B=\left( 
\begin{array}{cc}
0 & 1 \\ 
-1 & 0
\end{array}
\right) ,\ \ \ \ \Omega \left(x\right) =\left( 
\begin{array}{cc}
p(x) & q(x) \\ 
q(x) & -p(x)
\end{array}
\right) ,\ \ \ \ y=\left( 
\begin{array}{c}
y_{1}\left( x\right) \\ 
y_{2}\left( x\right)
\end{array}
\right), 
\]
$p\left( x\right),$ $q\left( x\right) $ are real valued functions in $L_{2}(0,\pi),$ 
$\lambda$ is a spectral parameter, $h_{1}$ and $h_{2}$ are real numbers, $h_{2}>0.$

In the finite interval, the inverse problems for Dirac differential equations system by different
spectral characteristics (e.g., one spectrum and norming constants or two spectra or spectral function) 
are solved in \cite{Ga-Dz, Dz, Al-Hr-My, My-Pu, Ho, Mo-Tr, MaS}. When spectral parameter contained 
in boundary condition, reconstruction of potential from spectral function is carried out in \cite{MaS}
and the uniqueness of inverse problem for Dirac operator according to Weyl function is  worked in \cite{Ami-Kes-Oz}. 
Inverse spectral problem for Dirac operator with potentials belong entrywise 
to $L_{p}(0,1)$; for some $p\in[1,\infty)$, was studied in \cite{Al-Hr-My} and in this work,
not only Gelfand-Levitan-Marchenko method but also Krein method \cite{Kr} was used.  Using
Weyl-Titschmarsh function, direct and inverse problems for Dirac type-system were investigated in \cite{Sa, Sa1}.
For weighted Dirac systems, inverse spectral problems was examined in \cite{Wa}.

On a positive half line, inverse scattering problem for a 
system of Dirac equations of order $2n$ is completely solved in
\cite{Gas} and when boundary condition involving spectral parameter, 
for Dirac operator, inverse scattering problem worked in \cite{Ma-Col1, Ma-Col2}. 
The applications of Dirac differential equations system has been widespread in various areas of physics, 
such as \cite{Sp1, Sp2, Th} ever since Dirac equation was discovered to be
associated with nonlinear wave equation in \cite{Ab1}.

This paper is organized as follows: In section 2, the asymptotic formulas of eigenvalues, 
eigenfunctions and normalizing numbers of the boundary value problem (\ref{1}), (\ref{2}) are investigated.
In section 3, by the eigenfunctions, completeness theorem 
is proved and expansion formula is obtained. Parseval equality
is given. In section 4, the main equation namely Gelfand-Levitan-Marchenko
type equation is derived. In section 5, we show that the boundary value problem 
(\ref{1}),(\ref{2}) can be uniquely determined from its eigenvalues and normalizing numbers.
Finally in section 6, the solution of inverse problem is obtained. Let's express
this more clearly. We can state the inverse problem for a system of
Dirac equations in the following way: let $\lambda_{n}$ and $\alpha_{n}$, $\left(n\in\mathbb{Z}\right)$ 
are respectively eigenvalues and normalizing numbers of boundary value problem (\ref{1}), (\ref{2}) and
the quantities $\left\{\lambda_{n},\alpha_{n}\right\}$, $\left(n\in\mathbb{Z}\right)$ are called spectral data. 
Knowing the spectral data $\left\{\lambda_{n},\alpha_{n}\right\}$, $\left(n\in\mathbb{Z}\right)$ to 
indicate a method of determining the potential $\Omega(x)$ and to find necessary and sufficient conditions for
$\lambda_{n}$ and $\alpha_{n}$, $\left(n\in\mathbb{Z}\right)$ to be the spectral data of a problem (\ref{1}),(\ref{2}),
for this, we derive differential equation, Parseval equality and boundary conditions. The
main theorem on the necessary and sufficient conditions for the solvability
of inverse problem is proved and then algorithm of the construction of the
function $\Omega(x)$ by spectral data is given. Note that throughout this paper, we 
use the following notation: $\tilde{\phi}$ denotes the transposed matrix of $\phi$.

\section{Asymptotic Formulas of Eigenvalues, Eigenfunction and Normalizing Numbers}

The inner product in Hilbert space $H=L_{2}(0,\pi;\mathbb{C}^{2})\oplus \mathbb{C}$ is defined by
\[
\left\langle Y,Z\right\rangle:=\int_{0}^{\pi}\left[y_{1}(x)\overline{z}_{1}(x)+
y_{2}(x)\overline{z}_{2}(x)\right]dx+\frac{1}{h_{2}}y_{3}\overline{z}_{3},
\] where 
\[
Y=\left( 
\begin{array}{c}
y_{1}(x) \\ 
y_{2}(x) \\ 
y_{3}
\end{array}
\right) \in H, \ \ \ \ \ \ \ \ Z=\left( 
\begin{array}{c}
z_{1}(x) \\ 
z_{2}(x) \\ 
z_{3}
\end{array}
\right) \in H. 
\]
Let us define 
\[
L(Y):=\left( 
\begin{array}{c}
l(y) \\ 
-h_{1}y_{1}(\pi)-h_{2}y_{2}(\pi)
\end{array}
\right) 
\]
with
\[
D(L)=\left\{ 
\begin{array}{c}
Y\left|\right.Y=\left( y_{1}(x),y_{2}(x),y_{3}\right)^{T} \in H,\
y_{1}(x),\ y_{2}(x)\in AC[0,\pi ], \\ 
y_{3}=y_{1}(\pi) ,\ l(y)\in L_{2}(0,\pi; \mathbb{C}^{2}),\
U(y)=0,\ V(y)=0
\end{array}
\right\} 
\]
where
\[
l(y)=\left(\begin{array}{c}
y_{2}^{'}+p(x)y_{1}+q(x)y_{2}\\
-y_{1}^{'}+q(x)y_{1}-p(x)y_{2}
\end{array}\right) . 
\]
Let $\varphi(x,\lambda)$ and $\psi(x,\lambda)$ be solutions of the system (\ref{1}) satisfying
the initial conditions
\begin{equation}\label{b}
\varphi(0,\lambda)=\left(\begin{array}{c}
0 \\
-1
\end{array}\right), \ \ \ \psi(\pi,\lambda)=\left(
\begin{array}{c}
h_{2} \\
-\lambda-h_{1}
\end{array}
\right). 
\end{equation} 
The solution $\varphi(x,\lambda)$ has the following representation (see \cite{Ami-Kes-Oz})
\begin{equation}\label{3}
\varphi(x,\lambda)=\varphi_{0}(x,\lambda)+\int_{0}^{x}K(x,t)\varphi_{0}(t,\lambda)dt,
\end{equation}
where $\varphi_{0}(x,\lambda)=\left(\begin{array}{c}
\sin\lambda x \\
-\cos\lambda x
\end{array}\right)$, $K_{ij}(x,.)\in L_{2}(0,\pi), \ i,j=1,2$ for fixed $x\in[0,\pi]$ and $K(x,t)$
is solution of the problem
\[
BK'_{x}(x,t)+K'_{t}(x,t)B=-\Omega(x)K(x,t),
\]
\begin{equation}\label{4}
\Omega(x)=K(x,x)B-BK(x,x),
\end{equation}
\[
K_{11}(x,0)=K_{21}(x,0)=0.
\] The formula (\ref{4}) gives the relation between the kernel $K(x,t)$
and the coefficient of $\Omega(x)$ of the equation (\ref{1}). 

The characteristic function $\Delta(\lambda)$ of the problem $L$ is 
\begin{equation}\label{5}
\Delta(\lambda):=W\left[\varphi(x,\lambda), \psi(x,\lambda)\right]=
\varphi_{2}(x,\lambda)\psi_{1}(x,\lambda)-\varphi_{1}(x,\lambda)\psi_{2}(x,\lambda),
\end{equation} 
where $W\left[\varphi(x,\lambda), \psi(x,\lambda)\right]$ is Wronskian of the solutions $\varphi(x,\lambda)$ and 
$\psi(x,\lambda)$ and independent of $x\in[0,\pi].$
The zeros of $\Delta(\lambda)$ coincide with the eigenvalues $\lambda_{n}$ of problem $L$. 
The functions $\varphi(x,\lambda)$ and $\psi(x,\lambda)$ are eigenfunctions and there exists a sequence $\beta_{n}$
such that
\begin{equation}\label{6}
\psi(x,\lambda_{n})=\beta_{n}\varphi(x,\lambda_{n}), \ \ \ \beta_{n}\neq 0.
\end{equation}
Normalizing numbers are
\[
\alpha_{n}:=\int_{0}^{\pi}\left(\left|\varphi_{1}(x,\lambda_{n})\right|^{2}+
\left|\varphi_{2}(x,\lambda_{n})\right|^{2}\right)dx+
\frac{1}{h_{2}}\left|\varphi_{1}(\pi,\lambda_{n})\right|^{2}.
\] 
The following relation holds (see \cite{Mam-Ak}):
\begin{equation}\label{7}
\dot{\Delta}(\lambda_{n})=\beta_{n}\alpha_{n},
\end{equation}
where $\dot{\Delta}(\lambda)=\frac{d}{d\lambda}\Delta(\lambda).$ 

\begin{theorem} \label{thm1}
i) The eigenvalues $\lambda_{n},$ $\left(n\in\mathbb{Z}\right)$ of boundary value problem (\ref{1}), (\ref{2}) are 
\begin{equation}\label{8}
\lambda_{n}=\lambda_{n}^{0}+\epsilon_{n}, \ \ \ \left\{\epsilon_{n}\right\}\in l_{2},
\end{equation}
where $\lambda_{n}^{0}=n$ are zeros of function $\lambda\sin\lambda\pi$.
For the large $n$, the eigenvalues are simple. 

ii) The eigenfunctions of the boundary value problem can be represented in the form
\begin{equation}\label{9}
\varphi(x,\lambda_{n})=\left(\begin{array}{c}
\sin nx \\
-\cos nx 
\end{array}
\right)+\left(\begin{array}{c}
\zeta_{n}^{(1)}(x) \\
\zeta_{n}^{(2)}(x) 
\end{array}
\right),
\end{equation}
where $\sum_{n=-\infty}^{\infty}\left\{\left|\zeta_{n}^{(1)}(x)\right|^{2}+\left|\zeta_{n}^{(2)}(x)\right|^{2}\right\}\leq C,$
in here $C$ is a positive number. 

iii)Normalizing numbers of the problem (\ref{1}), (\ref{2}) are as follows
\begin{equation}\label{10}
\alpha_{n}=\pi+\tau_{n}, \ \ \ \left\{\tau_{n}\right\}\in l_{2}.
\end{equation}
\end{theorem}

\begin{proof}
i) Using (\ref{3}), as $\left|\lambda\right|\rightarrow\infty$ uniformly in $x\in[0,\pi]$ the following asymptotic 
formulas are obtain:
\begin{equation}\label{11}
\begin{array}{c}
\varphi_{1}(x,\lambda)=\sin\lambda x+O\left(\frac{1}{\left|\lambda\right|}e^{\left|Im\lambda\right|x}\right), \\
\varphi_{2}(x,\lambda)=-\cos\lambda x+O\left(\frac{1}{\left|\lambda\right|}e^{\left|Im\lambda\right|x}\right).
\end{array}
\end{equation}
Substituting the asymptotic formulas (\ref{11}) into
\[
\Delta(\lambda)=\left(\lambda+h_{1}\right)\varphi_{1}(\pi,\lambda)+h_{2}\varphi_{2}(\pi,\lambda),
\] we get as $\left|\lambda\right|\rightarrow\infty$
\begin{equation}\label{12}
\Delta(\lambda)=\lambda\sin\lambda\pi+O\left(e^{\left|Im\lambda\right|\pi}\right).
\end{equation}
Denote 
\[
G_{\delta}:=\left\{\lambda:\ \left|\lambda-n\right|\geq\delta, \ \ \ n=0,\pm 1,\pm 2...\right\},
\] 
where $\delta$ is a sufficiently small positive number.
There exists a positive number $C_{\delta}$ such that (see \cite{Mar})
\begin{equation}\label{13}
\left|\sin\lambda\pi\right|\geq C_{\delta}e^{\left|Im\lambda\right|\pi}, \ \ \ \lambda\in G_{\delta}.
\end{equation}
On the other hand from (\ref{12}), as $\left|\lambda\right|\rightarrow\infty$
\begin{equation}\label{14}
\Delta(\lambda)-\lambda\sin\lambda\pi=O\left(e^{\left|Im\lambda\right|\pi}\right).
\end{equation} Therefore, on infinitely expanding contours 
\[
\Gamma_{n}=\left\{\lambda:\left|\lambda\right|=n+\frac{1}{2}\right\}
\] for sufficiently large $n$, using (\ref{13}) and (\ref{14}) we have
\[
\left|\Delta(\lambda)-\lambda\sin\lambda\pi\right|<\left|\lambda\right|\left|\sin\lambda\pi\right|, \ \ \lambda\in\Gamma_{n}.
\] Applying the Rouche theorem, it is obtained that the number of zeros of the
function $\left\{\Delta(\lambda)-\lambda\sin\lambda\pi\right\}+\lambda\sin\lambda\pi=\Delta(\lambda)$ 
inside the counter $\Gamma_{n}$ coincides
with the number of zeros of function $\lambda\sin\lambda\pi.$ Moreover, using the Rouche
theorem, there exist only one zero $\lambda_{n}$ of the function $\Delta(\lambda)$ in the circle 
$\gamma_{n}=\left\{\lambda:\left|\lambda-n\right|<\delta\right\}$ is concluded. Since $\delta>0$ is arbitrary, we have
\begin{equation}\label{15}
\lambda_{n}=n+\epsilon_{n}, \ \ \ \ \lim_{n=\pm\infty}\epsilon_{n}=0.
\end{equation} 
Substituting (\ref{15}) into (\ref{12}), we get $\sin\epsilon_{n}\pi=O\left(\frac{1}{n}\right),$ i.e.
$\epsilon_{n}=O\left(\frac{1}{n}\right).$ Thus $\left\{\epsilon_{n}\right\}\in l_{2}$ is found. 
For the large n, the eigenvalues are simple. In fact, $\alpha_{n}\beta_{n}=\dot{\Delta}(\lambda_{n})$ and 
$\alpha_{n}\neq 0$, $\beta_{n}\neq 0$, we get $\dot{\Delta}(\lambda_{n})\neq 0.$ 
\\

ii) Putting (\ref{8}) into (\ref{3}), we have
\[
\varphi(x,\lambda_{n})=\left(\begin{array}{c}
\sin nx \\
-\cos nx 
\end{array}
\right)+\left(\begin{array}{c}
\zeta_{n}^{(1)}(x) \\
\zeta_{n}^{(2)}(x) 
\end{array}
\right),
\] where 
\begin{eqnarray}\nonumber
\zeta_{n}^{(1)}(x)&=&\sin nx\left[\cos\epsilon_{n}x-1\right]+\cos nx\sin\epsilon_{n}x+ \\ \nonumber
&&+\int_{0}^{x}K_{11}(x,t)\sin(n+\epsilon_{n})tdt-\int_{0}^{x}K_{12}(x,t)\cos(n+\epsilon_{n})tdt, \nonumber
\end{eqnarray}
\begin{eqnarray}\nonumber
\zeta_{n}^{(2)}(x)&=&-\cos nx\left[\cos\epsilon_{n}x-1\right]+\sin nx\sin\epsilon_{n}x+ \\ \nonumber
&&+\int_{0}^{x}K_{21}(x,t)\sin(n+\epsilon_{n})tdt-\int_{0}^{x}K_{22}(x,t)\cos(n+\epsilon_{n})tdt.
\end{eqnarray} 
Since $K_{ij}(\pi,.)\in L_{2}(0,\pi)$, $i,j=1,2$, according to (see \cite{Mar}, pp.66) we get
\[
\left\{\int_{0}^{\pi}K_{i1}(\pi,t)\sin\lambda_{n}tdt\right\}\in l_{2}, \ \ \ \ \
\left\{\int_{0}^{\pi}K_{i2}(\pi,t)\cos\lambda_{n}tdt\right\}\in l_{2}.
\] Using these relations and $\left\{\epsilon_{n}\right\}\in l_{2}$, we obtain
\[\sup_{0\leq x\leq\pi}\sum_{n=-\infty}^{\infty}\left\{\left|\zeta_{n}^{(1)}(x)\right|^{2}+
\left|\zeta_{n}^{(2)}(x)\right|^{2}\right\}<+\infty.\] 
\\

iii) Using the definition of normalizing numbers and (\ref{9})
\begin{eqnarray}\nonumber
\alpha_{n}&=&\int_{0}^{\pi}\left(\left|\varphi_{1}(x,\lambda_{n})\right|^{2}+
\left|\varphi_{2}(x,\lambda_{n})\right|^{2}\right)dx+
\frac{1}{h_{2}}\left|\varphi_{1}(\pi,\lambda_{n})\right|^{2} \\ \nonumber
&=&\int_{0}^{\pi}dx+\tau_{n}=\pi+\tau_{n},
\end{eqnarray} where
\begin{eqnarray}\nonumber
\tau_{n}&=&\int_{0}^{\pi}\left(\zeta_{n}^{(1)}(x)\right)^{2}dx+\int_{0}^{\pi}\left(\zeta_{n}^{(2)}(x)\right)^{2}dx+
2\int_{0}^{\pi}\sin nx\zeta_{n}^{(1)}(x)dx- \\ \nonumber
&&-2\int_{0}^{\pi}\cos nx\zeta_{n}^{(2)}(x)dx+\frac{1}{h_{2}}\left(\zeta_{n}^{(1)}(\pi)\right)^{2}. \nonumber
\end{eqnarray} 
Since $\left\{\zeta_{n}^{(1)}(x)\right\}\in l_{2}$ and $\left\{\zeta_{n}^{(2)}(x)\right\}\in l_{2}$, we find 
$\left\{\tau_{n}\right\}\in l_{2}.$ The theorem is proved.
\end{proof}

\begin{proposition} \label{p1}
The specification of the eigenvalues $\lambda_{n}, \ (n\in\mathbb{Z})$ uniquely determines the
characteristic function $\Delta(\lambda)$ by formula 
\begin{equation}\label{a}
\Delta(\lambda)=-\pi(\lambda_{0}^{2}-\lambda^{2})\prod_{n=1}^{\infty}\frac{(\lambda_{n}^{2}-\lambda^{2})}{n^{2}}.
\end{equation}
\end{proposition}
\begin{proof}
Since the function $\Delta(\lambda)$ is entire function of order $1$, from Hadamard's theorem (see \cite{Le}),
using (\ref{12}) we obtain (\ref{a}).
\end{proof}

\section{Parseval Equality}

\begin{theorem} \label{thm2}
a) The system of eigenfunctions $\varphi(x,\lambda_{n})$, $(n\in \mathbb{Z})$ of boundary value problem (\ref{1}), (\ref{2}) 
is complete in space $L_{2}(0,\pi;\mathbb{C}^{2})$. 

b) Let $f(x)\in D(L).$  Then the expansion formula
\begin{equation}\label{16}
f(x)=\sum_{n=-\infty}^{\infty}a_{n}\varphi(x,\lambda_{n}),
\end{equation}
\[
a_{n}=\frac{1}{\alpha_{n}}\left\langle f(x), \varphi(x,\lambda_{n})\right\rangle
\] is valid and the series converges uniformly with respect to $x\in[0,\pi].$ 
For $f(x)\in L_{2}(0,\pi;\mathbb{C}^{2})$ series (\ref{16}) converges in $L_{2}(0,\pi;\mathbb{C}^{2})$, moreover,
Parseval equality holds
\begin{equation}\label{17}
\left\|f\right\|^{2}=\sum_{n=-\infty}^{\infty}\alpha_{n}\left|a_{n}\right|^{2}.
\end{equation}
\end{theorem}
\begin{proof}
a) Denote
\[
G(x,t,\lambda):=-\frac{1}{\Delta(\lambda)}\left\{\begin{array}{c}
\psi(x,\lambda)\tilde{\varphi}(t,\lambda), \ \ \ x\geq t \\
\varphi(x,\lambda)\tilde{\psi}(t,\lambda), \ \ \ t\geq x. 
\end{array}\right.
\] Consider the function
\begin{equation}\label{18}
Y(x,\lambda)=\int_{0}^{\pi}G(x,t,\lambda)f(t)dt 
\end{equation} 
which gives solution of boundary value problem
\begin{equation}\label{19}
\begin{array}{c}
BY'+\Omega(x)Y=\lambda Y+f(x), \\
Y_{1}(0,\lambda)=0, \\
\left(\lambda+h_{1}\right)Y_{1}(\pi,\lambda)+h_{2}Y_{2}(\pi,\lambda)=0.
\end{array}
\end{equation}
Using (\ref{6}) and (\ref{7}), we have
\[
\psi(x,\lambda_{n})=\frac{\dot{\Delta}(\lambda_{n})}{\alpha_{n}}\varphi(x,\lambda_{n}). 
\] Using this expression, we find
\begin{equation}\label{20}
\underset{\lambda =\lambda _{n}}{Res}Y(x,\lambda)=-\frac{1}{\alpha_{n}} \varphi(x,\lambda
_{n})\int_{0}^{\pi}\widetilde{\varphi}(t,\lambda _{n})f(t)dt.
\end{equation}
Let $f(x)\in L_{2}(0,\pi;\mathbb{C}^{2})$ be such that
\begin{equation}\label{21}
\left\langle f(x),\varphi(x,\lambda _{n})\right\rangle =\int_{0}^{\pi}
\tilde{\varphi}(t,\lambda_{n})f(t)dt=0. 
\end{equation}
Then, it follows from (\ref{20}) that $\underset{\lambda =\lambda _{n}}{Res}Y(x,\lambda )=0$.
Consequently $Y(x,\lambda )$ is entire function with
respect to $\lambda $ for each fixed $x\in \left[0,\pi \right] .$ Taking
into account 
\begin{equation}\label{22}
\left|\Delta(\lambda)\right|\geq C_{\delta}\left|\lambda\right|e^{\left|Im\lambda\right|\pi}
\end{equation}
and the following equalities being valid according to (\cite{Mar}, Lemma 1.3.1)
\begin{equation}\label{23}
\lim_{\left| \lambda \right| \rightarrow \infty }\max_{0\leq x\leq \pi }\exp
(-\left| Im\lambda\right|x)\left| \int_{0}^{x} \tilde{\varphi}
(t,\lambda)f(t)dt\right|=0,  
\end{equation}
\begin{equation}\label{24}
\lim_{\left|\lambda\right|\rightarrow\infty}\max_{0\leq x \leq\pi}\frac{1}{
\left|\lambda\right|}\exp(-\left|Im\lambda\right|\left(\pi-x\right)) 
\left|\int_{x}^{\pi}\tilde{\psi}(t,\lambda)f(t)dt\right|=0,  
\end{equation}
we find
\[
\lim_{\left|\lambda\right| \rightarrow \infty }\max_{0\leq x\leq \pi
}\left|Y(x,\lambda)\right| =0. 
\]
Hence $Y(x,\lambda)\equiv 0$ is obtained. From here and (\ref{19}), $f(x)=0$ a.e. on $(0,\pi)$. 
\\ 

b) Since $\varphi(x,\lambda )$ and $\psi (x,\lambda )$ are solution of the problem (\ref{1}), (\ref{2}),
\begin{eqnarray} \nonumber
Y(x,\lambda ) &=&-\frac{1}{\lambda \Delta (\lambda)}\psi (x,\lambda
)\int_{0}^{x}\left\{-\frac{\partial}{\partial t}\tilde{\varphi}(t,\lambda
)B+ \tilde{\varphi}(t,\lambda)\Omega(t)\right\} f(t)dt \\ \nonumber
&&-\frac{1}{\lambda \Delta(\lambda )}\varphi(x,\lambda)\int_{x}^{\pi}\left\{ - 
\frac{\partial }{\partial t}\widetilde{\psi}(t,\lambda )B+\tilde{\psi }
(t,\lambda )\Omega(t)\right\} f(t)dt \nonumber
\end{eqnarray}
can be written. Integrating by parts and using (\ref{5}),
\begin{equation} \label{25}
Y(x,\lambda)=-\frac{1}{\lambda}f(x)-\frac{1}{\lambda}Z(x,\lambda)  
\end{equation}
is obtained, where 
\begin{eqnarray} \nonumber
Z(x,\lambda)&=&\frac{1}{\Delta(\lambda)}\left\{\psi(x,\lambda)\int_{0}^{x} 
\tilde{\varphi}(t,\lambda)Bf^{\prime }(t)dt+ \varphi(x,\lambda)\int_{x}^{\pi} 
\tilde{\psi}(t,\lambda)Bf^{\prime}(t)dt\right. \\ \nonumber
&&+\left.\psi(x,\lambda)\int_{0}^{x}\tilde{\varphi}(t,\lambda)\Omega(t)f(t)dt+\varphi(x,\lambda)\int_{x}^{\pi} 
\tilde{\psi}(t,\lambda)\Omega(t)f(t)dt\right\}.  \nonumber
\end{eqnarray}
By applying (\ref{23}) and (\ref{24}), we have
\begin{equation}\label{26}
\lim_{\left|\lambda\right|\rightarrow\infty}\max_{0\leq x
\leq\pi}\left|Z(x,\lambda)\right|=0, \ \ \ \ \lambda\in G_{\delta}. 
\end{equation}
Now, we integrate of $Y(x,\lambda)$ with respect to $\lambda$ over the
contour $\Gamma_{N}$ with oriented counter clockwise as follows: 
\[
I_{n}(x)=\frac{1}{2\pi i}\oint_{\Gamma_{N}}Y(x,\lambda)d\lambda, 
\]
where 
$\Gamma_{N}=\left\{\lambda :\left| \lambda\right|=N+\frac{1}{2}\right\} $, 
$N$ is sufficiently large natural number. Using residue theorem, we get
\[
I_{n}(x)=\sum_{n=-N}^{N}
\underset{\lambda =\lambda _{n}}{Res}
Y(x,\lambda )
=-\sum_{n=-N}^{N}\frac{1}{\alpha _{n}}\varphi(x,\lambda_{n})\int_{0}^{\pi} 
\tilde{\varphi}(t,\lambda_{n})f(t)dt. 
\]
On the other hand, considering the equation (\ref{25})
\begin{equation}\label{27}
f(x)=\sum_{n=-N}^{N}a_{n}\varphi(x,\lambda _{n})+\epsilon_{N}(x) 
\end{equation}
is found, where 
\[
\epsilon _{N}(x)=-\frac{1}{2\pi i}\oint_{\Gamma _{N}}\frac{1}{\lambda}
Z(x,\lambda )d\lambda 
\]
and 
\[
a_{n}=\frac{1}{\alpha_{n}}\int_{0}^{\pi}\widetilde{\varphi}(t,\lambda_{n})f(t)dt. 
\]
It follows from (\ref{26}) that,
\[
\lim_{N \rightarrow \infty }\max_{0\leq x\leq \pi
}\left\vert \epsilon_{N}(x)\right\vert =0. 
\] Thus, by going over in (\ref{27}) to the limit as $N\rightarrow\infty$ the expansion formula (\ref{16}) is obtained. 
The system $\left\{\varphi(x,\lambda_{n})\right\}$, $(n\in \mathbb{Z})$ 
is complete and orthogonal in $L_{2}(0,\pi;\mathbb{C}^{2})$.
Therefore, it forms orthogonal basis in $L_{2}(0,\pi;\mathbb{C}^{2})$ and Parseval equality (\ref{17}) holds.  
\end{proof}

\section{Main Equation}
\begin{theorem} \label{thm3}
For each fixed $x\in(0,\pi]$ the kernel $K(x,t)$ satisfies the following equation:
\begin{equation}\label{28}
F(x,t)+K(x,t)+\int_{0}^{x}K(x,\xi)F(\xi,t)d\xi=0, \ \ \ \ 0<t<x,
\end{equation}
where
\begin{equation}\label{29}
F(x,t)=\sum_{n=-\infty}^{\infty}\left[\frac{1}{\alpha_{n}}\varphi_{0}(x,\lambda_{n})\tilde{\varphi}_{0}(t,\lambda_{n})-
\frac{1}{\pi}\varphi_{0}(x,\lambda_{n}^{0})\tilde{\varphi}_{0}(t,\lambda_{n}^{0})\right].
\end{equation} 
\end{theorem} 

\begin{proof}
Using the transformation operators (see \cite{Lev}), according to (\ref{3}) we have
\begin{equation}\label{31}
\varphi_{0}(x,\lambda)=\varphi(x,\lambda)+\int_{0}^{x}H(x,t)\varphi(t,\lambda)dt.
\end{equation}
It follows from (\ref{3}) and (\ref{31}) that
\begin{eqnarray}\nonumber
\lefteqn{\sum_{n=-N}^{N}\frac{1}{\alpha_{n}}\varphi(x,\lambda_{n})\tilde{\varphi}_{0}(t,\lambda_{n})=} \\ \nonumber
&=&\sum_{n=-N}^{N}\frac{1}{\alpha_{n}}\varphi_{0}(x,\lambda_{n})\tilde{\varphi}_{0}(t,\lambda_{n})+
\int_{0}^{x}K(x,\xi)\sum_{n=-N}^{N}\frac{1}{\alpha_{n}}\varphi_{0}(\xi,\lambda_{n})\tilde{\varphi}_{0}(t,\lambda_{n})d\xi
\end{eqnarray}
and
\begin{eqnarray}\nonumber
\lefteqn{\sum_{n=-N}^{N}\frac{1}{\alpha_{n}}\varphi(x,\lambda_{n})\tilde{\varphi}_{0}(t,\lambda_{n})=} \\ \nonumber
&=&\sum_{n=-N}^{N}\frac{1}{\alpha_{n}}\varphi(x,\lambda_{n})\tilde{\varphi}(t,\lambda_{n})+
\sum_{n=-N}^{N}\frac{1}{\alpha_{n}}\varphi(x,\lambda_{n})\int_{0}^{t}\tilde{\varphi}(\xi,\lambda_{n})\tilde{H}(t,\xi)d\xi.
\end{eqnarray}
Using the last two equalities, we obtain
\begin{eqnarray}\nonumber
\lefteqn{\sum_{n=-N}^{N}\left[\frac{1}{\alpha_{n}}\varphi(x,\lambda_{n})\tilde{\varphi}(t,\lambda_{n})-
\frac{1}{\pi}\varphi_{0}(x,\lambda_{n}^{0})\tilde{\varphi}_{0}(t,\lambda_{n}^{0})\right]=}
\\ \nonumber
&=&\sum_{n=-N}^{N}\left[\frac{1}{\alpha_{n}}\varphi_{0}(x,\lambda_{n})\tilde{\varphi}_{0}(t,\lambda_{n})-
\frac{1}{\pi}\varphi_{0}(x,\lambda_{n}^{0})\tilde{\varphi}_{0}(t,\lambda_{n}^{0})\right]+
\\ \nonumber
&&+\int_{0}^{x}K(x,\xi)\sum_{n=-N}^{N}\left[\frac{1}{\alpha_{n}}\varphi_{0}(\xi,\lambda_{n})\tilde{\varphi}_{0}(t,\lambda_{n})-
\frac{1}{\pi}\varphi_{0}(\xi,\lambda_{n}^{0})\tilde{\varphi}_{0}(t,\lambda_{n}^{0})\right]d\xi+
\\ \nonumber
&&+\int_{0}^{x}K(x,\xi)\sum_{n=-N}^{N}\frac{1}{\pi}\varphi_{0}(\xi,\lambda_{n}^{0})
\tilde{\varphi}_{0}(\xi,\lambda_{n}^{0})d\xi- \\ \nonumber
&&-\sum_{n=-N}^{N}\frac{1}{\alpha_{n}}\varphi(x,\lambda_{n})\int_{0}^{t}\tilde{\varphi}
(\xi,\lambda_{n})\tilde{H}(t,\xi)d\xi \nonumber
\end{eqnarray}
or
\begin{equation}\label{32}
\Phi_{n}(x,t)=I_{N1}(x,t)+I_{N2}(x,t)+I_{N3}(x,t)+I_{N4}(x,t),
\end{equation} 
where
\[
\Phi_{n}(x,t):=\sum_{n=-N}^{N}\left[\frac{1}{\alpha_{n}}\varphi(x,\lambda_{n})\tilde{\varphi}(t,\lambda_{n})-
\frac{1}{\pi}\varphi_{0}(x,\lambda_{n}^{0})\tilde{\varphi}_{0}(t,\lambda_{n}^{0})\right],
\]
\[
I_{N1}(x,t):=\sum_{n=-N}^{N}\left[\frac{1}{\alpha_{n}}\varphi_{0}(x,\lambda_{n})\tilde{\varphi}_{0}(t,\lambda_{n})-
\frac{1}{\pi}\varphi_{0}(x,\lambda_{n}^{0})\tilde{\varphi}_{0}(t,\lambda_{n}^{0})\right],
\]
\[
I_{N2}(x,t):=\int_{0}^{x}K(x,\xi)\sum_{n=-N}^{N}\left[\frac{1}{\alpha_{n}}\varphi_{0}
(\xi,\lambda_{n})\tilde{\varphi}_{0}(t,\lambda_{n})-
\frac{1}{\pi}\varphi_{0}(\xi,\lambda_{n}^{0})\tilde{\varphi}_{0}(t,\lambda_{n}^{0})\right]d\xi,
\]
\[
I_{N3}(x,t):=\int_{0}^{x}K(x,\xi)\sum_{n=-N}^{N}\frac{1}{\pi}\varphi_{0}(\xi,\lambda_{n}^{0})
\tilde{\varphi}_{0}(\xi,\lambda_{n}^{0})d\xi,
\]
\[
I_{N4}(x,t):=\sum_{n=-N}^{N}\frac{1}{\alpha_{n}}\varphi(x,\lambda_{n})\int_{0}^{t}
\tilde{\varphi}(\xi,\lambda_{n})\tilde{H}(t,\xi)d\xi.
\] 
Let $f(x)$ be an absolutely continuous function. Then using expansion formula (\ref{16}),
\[
\lim_{N\rightarrow\infty}\int_{0}^{\pi}\Phi_{N}(x,t)f(t)dt=
\lim_{N\rightarrow\infty}\int_{0}^{\pi}\sum_{n=-N}^{N}\frac{1}{\alpha_{n}}
\varphi(x,\lambda_{n})\tilde{\varphi}(t,\lambda_{n})f(t)dt-
\]
\begin{equation}\label{33}
-\lim_{N\rightarrow\infty}\int_{0}^{\pi}\sum_{n=-N}^{N}
\frac{1}{\pi}\varphi_{0}(x,\lambda_{n}^{0})\tilde{\varphi}_{0}(t,\lambda_{n}^{0})f(t)dt=0,
\ \ \ \ \ \ \ \ \ \ \ \ \ \ \ \ \ \ \ \ 
\end{equation}
\begin{eqnarray}\nonumber
\lefteqn{\lim_{N\rightarrow\infty}\int_{0}^{\pi}I_{N1}(x,t)f(t)dt=} \\ \nonumber
&=&\lim_{N\rightarrow\infty}\int_{0}^{\pi}
\sum_{n=-N}^{N}\left[\frac{1}{\alpha_{n}}\varphi_{0}(x,\lambda_{n})\tilde{\varphi}_{0}(t,\lambda_{n})
-\frac{1}{\pi}\varphi_{0}(x,\lambda_{n}^{0})\tilde{\varphi}_{0}(t,\lambda_{n}^{0})\right]f(t)dt \nonumber
\end{eqnarray}
\begin{equation}\label{34}
=\int_{0}^{\pi}F(x,t)f(t)dt, \ \ \ \ \ \ \ \ \ \ \ \ \ \ \ \ \ \ \ \ \ \ \ \ \ \ \ \ \ \ \ \ \ \ \ \ \ \ \ \ \ \ \ \ \ \ 
\ \ \ \ \ \ \ \ \ \ \ \ \ \ \
\end{equation}
\[
\lim_{N\rightarrow\infty}\int_{0}^{\pi}I_{N2}(x,t)f(t)dt=
\lim_{N\rightarrow\infty}\int_{0}^{\pi}\int_{0}^{x}K(x,\xi)\sum_{n=-N}^{N}\left[\frac{1}{\alpha_{n}}
\varphi_{0}(\xi,\lambda_{n})\tilde{\varphi}_{0}(t,\lambda_{n})-\right.
\]
\begin{equation}\label{35}
\left.-\frac{1}{\pi}\varphi_{0}(\xi,\lambda_{n}^{0})\tilde{\varphi}_{0}(t,\lambda_{n}^{0})\right]f(t)d\xi dt
=\int_{0}^{\pi}\left(\int_{0}^{x}K(x,\xi)F(x,\xi)d\xi\right)f(t)dt,
\end{equation}
\begin{eqnarray}\nonumber
\lefteqn{\lim_{N\rightarrow\infty}\int_{0}^{\pi}I_{N3}(x,t)f(t)dt=} \\ \nonumber
&=&\lim_{N\rightarrow\infty}\int_{0}^{\pi}
\int_{0}^{x}K(x,\xi)\sum_{n=-N}^{N}\frac{1}{\pi}\varphi_{0}(\xi,\lambda_{n}^{0})
\tilde{\varphi}_{0}(\xi,\lambda_{n}^{0})f(t)d\xi dt \nonumber
\end{eqnarray}
\begin{equation}\label{36}
=\int_{0}^{x}K(x,\xi)f(\xi)d\xi,\ \ \ \ \ \ \ \ \ \ \ \ \ \ \ \ \ \ \ \ \ \ \ \ \ \ \ \ \ \ \ \ \ \ \ \ \ \ \ \ \ \ \ \ \ \ \
\end{equation}
\begin{eqnarray}\nonumber
\lefteqn{\lim_{N\rightarrow\infty}\int_{0}^{\pi}I_{N4}(x,t)f(t)dt=}\\ \nonumber
&=&\lim_{N\rightarrow\infty}\int_{0}^{\pi}
\left(\sum_{n=-N}^{N}\frac{1}{\alpha_{n}}\varphi(x,\lambda_{n})
\int_{0}^{t}\tilde{\varphi}(\xi,\lambda_{n})\tilde{H}(t,\xi)d\xi\right)f(t)dt \nonumber
\end{eqnarray}
\begin{equation}\label{37}
=\int_{x}^{\pi}H(t,x)f(t)dt. \ \ \ \ \ \ \ \ \ \ \ \ \ \ \ \ \ \ \ \ \ \ \ \ \ \ \ \ \ \ \ \ \ \ \ \ \ \ \ \ \ \ \ \ \
\ \ \ \ \ \
\end{equation}
We put $K(x,t)=H(x,t)=0$ for $x<t.$ Using (\ref{32}), (\ref{33}), (\ref{34}), (\ref{35}), (\ref{36}) and (\ref{37}) 
we obtain
\begin{eqnarray} \nonumber
\lefteqn{\int_{0}^{\pi}F(x,t)f(t)dt+\int_{0}^{\pi}\left(\int_{0}^{x}K(x,\xi)F(\xi,t)d\xi\right)f(t)dt+} \\ \nonumber
&&+\int_{0}^{x}K(x,\xi)f(\xi)d\xi-\int_{x}^{\pi}H(t,x)f(t)dt=0. \nonumber
\end{eqnarray} 
Since $f(x)$ can be chosen arbitrarily,
\[
F(x,t)+K(x,t)+\int_{0}^{x}K(x,\xi)F(\xi,t)d\xi-H(t,x)=0.
\] When $t<x$, this equality implies (\ref{28}).
\end{proof}

\section{Theorem for the Solution of the Inverse Problem}

\begin{lemma} \label{lm1}
For each fixed $x\in (0,\pi]$ main equation (\ref{28}) has a unique solution $K(x,t)\in L_{2}(0,x)$.
\end{lemma}

\begin{proof}
It suffices to prove that homogeneous equation
\[
g(t)+\int_{0}^{x}g(s)F(s,t)ds=0
\] has only trivial solution $g(t)=0.$ 
Let $g(t)=\left(g_{1}(t),\ \ g_{2}(t)\right)\in L_{2}(0,x)$ 
be a solution of the above equation and $g(t)=0$ for $t \in (x,\pi)$. Then  
\[
\left(g(t),\ g(t)\right)+\left(\int_{0}^{x}g(s)F(s,t)ds,\ g(t)\right)=0.
\]
Using the expression (\ref{29}), we get
\begin{eqnarray} \nonumber
\lefteqn{\int_{0}^{x}\left(g_{1}^{2}(t)+g_{2}^{2}(t)\right)dt+} \\ \nonumber
&&+\sum_{n=-\infty}^{\infty}\frac{1}{\alpha_{n}} \nonumber
\int_{0}^{x}\left[\int_{0}^{x}\left(g_{1}(s)\sin\lambda_{n}(s)-g_{2}(s)\cos\lambda_{n}(s)\right)ds\right]
g_{1}(t)\sin\lambda_{n}tdt \\ \nonumber
&&-\sum_{n=-\infty}^{\infty}\frac{1}{\alpha_{n}}
\int_{0}^{x}\left[\int_{0}^{x}\left(g_{1}(s)\sin\lambda_{n}(s)-g_{2}(s)\cos\lambda_{n}(s)\right)ds\right]
g_{2}(t)\cos\lambda_{n}tdt \\ \nonumber
&&-\sum_{n=-\infty}^{\infty}\frac{1}{\pi}
\int_{0}^{x}\left[\int_{0}^{x}\left(g_{1}(s)\sin\lambda_{n}^{0}(s)-g_{2}(s)\cos\lambda_{n}^{0}(s)\right)ds\right]
g_{1}(t)\sin\lambda_{n}^{0}tdt \\ \nonumber
&&+\sum_{n=-\infty}^{\infty}\frac{1}{\pi}
\int_{0}^{x}\left[\int_{0}^{x}\left(g_{1}(s)\sin\lambda_{n}^{0}(s)-g_{2}(s)\cos\lambda_{n}^{0}(s)\right)ds\right]
g_{2}(t)\cos\lambda_{n}^{0}tdt  \nonumber
\end{eqnarray}
\[
=\int_{0}^{x}\left(g_{1}^{2}(t)+g_{2}^{2}(t)\right)dt+\sum_{n=-\infty}^{\infty}\frac{1}{\alpha_{n}}
\left(\int_{0}^{x}\left[g_{1}(t)\sin\lambda_{n}t-g_{2}(t)\cos\lambda_{n}t\right]dt\right)^{2} 
\]
\[
-\sum_{n=-\infty}^{\infty}\frac{1}{\pi}
\left(\int_{0}^{x}\left[g_{1}(t)\sin\lambda_{n}^{0}t-g_{2}(t)\cos\lambda_{n}^{0}t\right]dt\right)^{2}=0. 
\]
Thus, it follows from the last relation that
\[
\int_{0}^{x}\left(g_{1}^{2}(t)+g_{2}^{2}(t)\right)dt+\sum_{n=-\infty}^{\infty}\frac{1}{\alpha_{n}}
\left(\int_{0}^{x}g(t)\varphi_{0}(t,\lambda_{n})dt\right)^2-
\]
\[
-\sum_{n=-\infty}^{\infty}\frac{1}{\pi}
\left(\int_{0}^{x}g(t)\varphi_{0}(t,\lambda_{n}^{0})dt\right)^2=0.
\]
 Using Parseval equality,
\[
\left\|g\right\|^{2}=\sum_{n=-\infty}^{\infty}\frac{1}{\pi}
\left(\int_{0}^{x}g(t)\varphi_{0}(t,\lambda_{n}^{0})dt\right)^{2}
\] we have
\[
\sum_{n=-\infty}^{\infty}\frac{1}{\alpha_{n}}\left(\int_{0}^{x}g(t)\varphi_{0}(t,\lambda_{n})dt\right)^{2}=0.
\] Since $\alpha_{n}>0$ and from statement a) in Theorem \ref{thm2}, the system $\left\{\varphi_{0}(t,\lambda_{n})\right\}$, 
$(n\in \mathbb{Z})$ is complete in $L_{2}(0,\pi;\mathbb{C}^{2})$, we obtain $g(t)=0.$
\end{proof}

\begin{theorem} \label{thm4}
Let $L(\Omega(x), h_{1}, h_{2})$ and $\hat{L}(\hat{\Omega}(x), \hat{h}_{1}, \hat{h}_{2})$ 
be two boundary value problems and 
\[
\lambda_{n}=\hat{\lambda}_{n}, \ \ \ \ \ \alpha_{n}=\hat{\alpha}_{n}, \ \ \left(n\in\mathbb{Z}\right).
\] Then 
\[
\Omega(x)=\hat{\Omega}(x) \ a.e.\ on \ \left(0,\pi\right), \ \ h_{1}=\hat{h}_{1}, \ \ h_{2}=\hat{h}_{2}.
\]
\end{theorem}

\begin{proof}
According to (\ref{29}), $F(x,t)=\hat{F}(x,t)$. Then, from the main 
equation (\ref{28}), we have $K(x,t)=\hat{K}(x,t)$. It follows from (\ref{4}) that 
$\Omega(x)=\hat{\Omega}(x)$ a.e. on $\left(0,\pi\right)$. Taking into account (\ref{3}),
we find $\varphi(x,\lambda_{n})=\hat{\varphi}(x,\lambda_{n})$. In consideration of (\ref{12}), 
we get $\dot{\Delta}(\lambda_{n})\equiv\widehat{\dot{\Delta}}(\lambda_{n})$ and from 
(\ref{7})$, \ \beta_{n}=\hat{\beta}_{n}$. 
Thus, using (\ref{b}) and (\ref{6}), we obtain 
$h_{1}=\hat{h}_{1}$, \ \ $h_{2}=\hat{h}_{2}$.
\end{proof}

\section{Solution of Inverse Problem}

Let the real numbers $\left\{\lambda_{n},\alpha_{n}\right\}$, $(n\in\mathbb{Z})$ 
of the form (\ref{8}) and (\ref{10}) be given. Using these numbers, we construct the function 
$F(x,t)$ by the formulas (\ref{29}) and determine $K(x,t)$ 
from the main equation (\ref{28}). 

Now, let us construct the functions $\varphi(x,\lambda)$ by the formula (\ref{3}),
$\Omega(x)$ by the formula (\ref{4}), $\Delta(\lambda)$ by the formula (\ref{a}) and $\beta_{n}$
by the formula (\ref{7}) respectively, i.e.,
\[
\varphi(x,\lambda):=\varphi_{0}(x,\lambda)+\int_{0}^{x}K(x,t)\varphi(t,\lambda)dt,
\]
\[
\Omega(x):=K(x,x)B-BK(x,x),
\]
\[
\Delta(\lambda):=-\pi(\lambda_{0}^{2}-\lambda^{2})\prod_{n=1}^{\infty}\frac{(\lambda_{n}^{2}-\lambda^{2})}{n^{2}}, 
\]
\[
\beta_{n}:=\frac{\dot{\Delta}(\lambda_{n})}{\alpha_{n}}\neq 0.
\]
The function $F(x,t)$ can be rewritten as follows:
\[
F(x,t)=\frac{1}{2}\left[a(x-t)+a(x+t)T\right],
\] where 
\[
a(x)=\sum_{n=-\infty}^{\infty}\left[\frac{1}{\alpha_{n}}\left(
\begin{array}{cc}
\cos\lambda_{n}x & -\sin\lambda_{n}x \\
\sin\lambda_{n}x & \cos\lambda_{n}x
\end{array}\right)-\frac{1}{\pi}\left(
\begin{array}{cc}
\cos nx & -\sin nx \\
\sin nx & \cos nx
\end{array}\right)\right]
\] and $T=\left(\begin{array}{cc}
-1 & 0 \\
0 & 1
\end{array}\right)$. It is shown analogously in (\cite{Yur}, Lemma 1.5.4) that the function $a(x)\in W_{2}^{1}[0,2\pi]$.
  
\subsection{Derivation of the Differential Equation}

\begin{theorem}\label{thm5}
The relations hold: 
\begin{equation}\label{38}
B\varphi'(x,\lambda)+\Omega(x)\varphi(x,\lambda)=\lambda\varphi(x,\lambda),
\end{equation}
\begin{equation}\label{39}
\varphi_{1}(0,\lambda)=0, \ \ \ \ \varphi_{2}(0,\lambda)=-1.
\end{equation}
\end{theorem}
\begin{proof}
Differentiating on $x$ and $y$ the main equation (\ref{28}) respectively, we get
\begin{equation}\label{40}
F'_{x}(x,t)+K'_{x}(x,t)+K(x,x)F(x,t)+\int_{0}^{x}K'_{x}(x,\xi)F(\xi,t)d\xi=0,
\end{equation} 
\begin{equation}\label{41}
F'_{t}(x,t)+K'_{t}(x,t)+\int_{0}^{x}K(x,\xi)F'_{t}(\xi,t)d\xi=0.
\end{equation}
It follows from (\ref{29}) that
\begin{equation}\label{42}
F'_{t}(x,t)B+BF'_{x}(x,t)=0.
\end{equation}
Since $F(x,0)BS=0$, where $S=\left(\begin{array}{c}
0 \\
-1
\end{array}\right)$, using the main equation (\ref{28}), we obtain
\begin{equation}\label{43}
K(x,0)BS=0,
\end{equation}
or 
\[
K_{11}(x,0)=K_{21}(x,0)=0.
\]
Multiplying the equation (\ref{40}) on the left by $B$, we get
\[
BF'_{x}(x,t)+BK'_{x}(x,t)+BK(x,x)F(x,t)+\int_{0}^{x}BK'_{x}(x,\xi)F(\xi,t)d\xi=0,
\]
and multiplying the equation (\ref{41}) on the right by $B$, we find
\[
F'_{t}(x,t)B+K'_{t}(x,t)B+\int_{0}^{x}K(x,\xi)F'_{t}(\xi,t)Bd\xi=0.
\] 
Adding the last two equalities and using (\ref{42})
\[
BK'_{x}(x,t)+BK(x,x)F(x,t)+\int_{0}^{x}BK'_{x}(x,\xi)F(\xi,t)d\xi=
\]
\begin{equation}\label{44}
=-K'_{t}(x,t)B+\int_{0}^{x}K(x,\xi)BF'_{\xi}(\xi,t)d\xi\equiv I(x,t).
\end{equation}
Integrating by parts in (\ref{44}) and from (\ref{43}),
\begin{equation}\label{45}
I(x,t)=-K'_{t}(x,t)B+K(x,x)BF(x,t)-\int_{0}^{x}K'_{\xi}(x,\xi)BF(\xi,t)d\xi
\end{equation}
is obtained. Substituting (\ref{45}) into (\ref{44}), we have
\[
BK'_{x}(x,t)+BK(x,x)F(x,t)+K'_{t}(x,t)B-K(x,x)BF(x,t)+
\]
\begin{equation}\label{46}
+\int_{0}^{x}\left[BK'_{x}(x,\xi)+K'_{\xi}(x,\xi)B\right]F(\xi,t)d\xi=0.
\end{equation}
Multiplying (\ref{28}) on the left by $\Omega(x)$ in form of (\ref{4}) and adding to (\ref{46}),
\[
BK'_{x}(x,t)+K'_{t}(x,t)B+\Omega(x)K(x,t)+
\]
\begin{equation}\label{47}
+\int_{0}^{x}\left[BK'_{x}(x,\xi)+K'_{\xi}(x,\xi)B+\Omega(x)K(x,\xi)\right]F(\xi,t)d\xi=0
\end{equation}
is obtained. Setting
\[
J(x,t):=BK'_{x}(x,t)+K'_{t}(x,t)B+\Omega(x)K(x,t),
\] we can rewrite equation (\ref{47}) as follows
\begin{equation}\label{48}
J(x,t)+\int_{0}^{x}J(x,\xi)F(\xi,t)d\xi=0.
\end{equation}
According to Lemma \ref{lm1}, homogeneous equation (\ref{47}) has only trivial solution, i.e.,
\begin{equation}\label{49}
BK'_{x}(x,t)+K'_{t}(x,t)B+\Omega(x)K(x,t)=0, \ \ 0<t<x.
\end{equation}

Now, differentiating (\ref{3}) and multiplying on the left by $B$, we have
\[
B\varphi'_{x}(x,\lambda)=\lambda B\left(\begin{array}{c}
\cos\lambda x \\
\sin\lambda x
\end{array}
\right)+BK(x,x)\left(\begin{array}{c}
\sin\lambda x \\
-\cos\lambda x
\end{array}
\right)+ 
\]
\begin{equation}\label{50}
+\int_{0}^{x}BK'_{x}(x,t)\left(\begin{array}{c}
\sin\lambda t \\
-\cos\lambda t
\end{array}
\right)dt. 
\end{equation}
On the other hand, multiplying (\ref{3}) on the left by $\lambda$ and integrating by parts 
and using (\ref{43}), we find
\[
\lambda\varphi(x,\lambda)=\lambda B\left(\begin{array}{c}
\cos\lambda x \\
\sin\lambda x
\end{array}
\right)+K(x,x)B\left(\begin{array}{c}
\sin\lambda x \\
-\cos\lambda x
\end{array}
\right)- 
\]
\begin{equation}\label{51}
-\int_{0}^{x}K'_{t}(x,t)B\left(\begin{array}{c}
\sin\lambda t \\
-\cos\lambda t
\end{array}
\right)dt.
\end{equation}
It follows from (\ref{50}) and (\ref{51}) that 
\begin{eqnarray} \nonumber
\lambda\varphi(x,\lambda)&=&B\varphi'_{x}(x,\lambda)+\left[K(x,x)B-BK(x,x)\right]\left(\begin{array}{c}
\sin\lambda x \\
-\cos\lambda x
\end{array}
\right)- \\ \nonumber
&&-\int_{0}^{x}\left[K'_{t}(x,t)B+BK'_{x}(x,t)\right]\left(\begin{array}{c}
\sin\lambda t \\
-\cos\lambda t
\end{array}
\right)dt. \nonumber
\end{eqnarray} 
Taking into account (\ref{4}) and (\ref{49}),
\[
B\varphi'_{x}(x,\lambda)+\Omega(x)\varphi(x,\lambda)=\lambda\varphi(x,\lambda)
\] is obtained. For $x=0$, from (\ref{3}) we get (\ref{39}). 
\end{proof}

\subsection{Derivation of Parseval Equality}
\begin{theorem} \label{thm6}
For any $g(x)\in L_{2}(0,\pi;\mathbb{C}^{2})$, the following relation holds: 
\begin{equation}\label{52}
\left\|g\right\|_{L_{2}}=\sum_{n=-\infty}^{\infty}\frac{1}{\alpha_{n}}
\left(\int_{0}^{\pi}\tilde{\varphi}(t,\lambda_{n})g(t)dt\right)^{2}.
\end{equation}
\end{theorem}
\begin{proof}
Denote
\[
Q(\lambda):=\int_{0}^{\pi}\tilde{\varphi}(t,\lambda)g(t)dt.
\] It follows from (\ref{3}) that
\[
Q(\lambda):=\int_{0}^{\pi}\tilde{\varphi}_{0}(t,\lambda)h(t)dt, 
\] where
\begin{equation}\label{53}
h(t)=g(t)+\int_{t}^{\pi}\tilde{K}(s,t)g(s)ds.
\end{equation}
Similarly, in view of (\ref{31}), 
\begin{equation}\label{54}
g(t)=h(t)+\int_{x}^{\pi}\tilde{H}(s,t)h(s)ds,
\end{equation}
where for the kernel $H(x,t)$, we have the identity
\begin{equation}\label{55}
\tilde{H}(x,t)=F(t,x)+\int_{0}^{x}K(x,\xi)F(\xi,t)d\xi.
\end{equation}
Using (\ref{53}), we get
\[
\int_{0}^{\pi}F(x,t)h(t)dt=\int_{0}^{\pi}F(x,t)\left[g(t)+\int_{t}^{\pi}
\tilde{K}(u,t)g(u)du\right]dt=
\]
\[
=\int_{0}^{\pi}\left[F(x,t)+\int_{0}^{t}F(x,u)\tilde{K}(t,u)du\right]g(t)dt
\]
\[
=\int_{0}^{x}\left[F(x,t)+\int_{0}^{t}F(x,u)\tilde{K}(t,u)du\right]g(t)dt+
\]
\[
+\int_{x}^{\pi}\left[F(x,t)+\int_{0}^{t}F(x,u)\tilde{K}(t,u)du\right]g(t)dt.
\]
It follows from (\ref{28}) and (\ref{55}) that
\begin{equation}\label{56}
\int_{0}^{\pi}F(x,t)h(t)dt=\int_{0}^{x}H(x,t)g(t)dt-\int_{x}^{\pi}\tilde{K}(t,x)g(t)dt.
\end{equation}
From (\ref{29}) and Parseval equality, we obtain
\begin{eqnarray}\nonumber
\lefteqn{\int_{0}^{\pi}\left[h_{1}^{2}(t)+h_{2}^{2}(t)\right]dt+\int_{0}^{\pi}\int_{0}^{\pi}\tilde{h}(x)F(x,t)h(t)dtdx=} 
\\ \nonumber
&=&\int_{0}^{\pi}\left[h_{1}^{2}(t)+h_{2}^{2}(t)\right]dt+\sum_{n=-\infty}^{\infty}
\frac{1}{\alpha_{n}}\left(\int_{0}^{\pi}\tilde{\varphi}_{0}(t,\lambda_{n})h(t)dt\right)^{2}- \\ \nonumber
&&-\sum_{n=-\infty}^{\infty}\frac{1}{\pi}
\left(\int_{0}^{\pi}\tilde{\varphi}_{0}(t,\lambda_{n}^{0})h(t)dt\right)^{2} \\ \nonumber
&=&\sum_{n=-\infty}^{\infty}\frac{1}{\alpha_{n}}
\left(\int_{0}^{\pi}\tilde{\varphi}_{0}(t,\lambda_{n})h(t)dt\right)^{2}=
\sum_{n=-\infty}^{\infty}\frac{Q^{2}(\lambda_{n})}{\alpha_{n}}. \nonumber
\end{eqnarray}
Taking into account (\ref{56}), we have
\begin{eqnarray} \nonumber
\lefteqn{\sum_{n=-\infty}^{\infty}\frac{Q^{2}(\lambda_{n})}{\alpha_{n}}=
\int_{0}^{\pi}\left[h_{1}^{2}(t)+h_{2}^{2}(t)\right]dt+} \\ \nonumber
&&+\int_{0}^{\pi}\tilde{h}(x)\left(\int_{0}^{x}H(x,t)g(t)dt\right)dx-
\int_{0}^{\pi}\tilde{h}(x)\left(\int_{x}^{\pi}\tilde{K}(t,x)g(t)dt\right)dx \\ \nonumber
&=&\int_{0}^{\pi}\left[h_{1}^{2}(t)+h_{2}^{2}(t)\right]dt+
\int_{0}^{\pi}\left(\int_{t}^{\pi}\tilde{h}(x)H(x,t)dx\right)g(t)dt- \\ \nonumber
&&-\int_{0}^{\pi}\tilde{h}(x)\left(\int_{x}^{\pi}\tilde{K}(t,x)g(t)dt\right)dx, \nonumber
\end{eqnarray}
whence by formulas (\ref{53}) and (\ref{54}),
\begin{eqnarray} \nonumber
\lefteqn{\sum_{n=-\infty}^{\infty}\frac{Q^{2}(\lambda_{n})}{\alpha_{n}}=
\int_{0}^{\pi}\left[h_{1}^{2}(t)+h_{2}^{2}(t)\right]dt+} \\ \nonumber
&&+\int_{0}^{\pi}\left(\tilde{g}(t)-\tilde{h}(t)\right)g(t)dt-
\int_{0}^{\pi}\tilde{h}(x)\left(h(x)-g(x)\right)dx
\\ \nonumber
&=&\int_{0}^{\pi}\left(g_{1}^{2}(t)+g_{2}^{2}(t)\right)dt \nonumber
\end{eqnarray}
is obtained, i.e., the relation (\ref{52}) is valid.
\end{proof}

\begin{corollary} \label{c1}
For any function $f(x)$ and $g(x)\in L_{2}(0,\pi;\mathbb{C}^{2})$, the relation holds:
\begin{equation}\label{57}
\int_{0}^{\pi}\tilde{g}(x)f(x)dx=\sum_{n=-\infty}^{\infty}\frac{1}{\alpha_{n}}
\left(\int_{0}^{\pi}\tilde{g}(t)\varphi(t,\lambda_{n})dt\right)\left(\int_{0}^{\pi}
\tilde{\varphi}(t,\lambda_{n})f(t)dt\right).
\end{equation}
\end{corollary}
\begin{lemma} \label{lm2}
For any $f(x)\in W_{2}^{1}[0,\pi]$, the expansion formula
\begin{equation}\label{58}
f(x)=\sum_{n=-\infty}^{\infty}c_{n}\varphi(x,\lambda_{n})
\end{equation}
is valid, where
\[
c_{n}=\frac{1}{\alpha_{n}}\int_{0}^{\pi}\tilde{\varphi}(x,\lambda_{n})f(x)dx.
\]
\end{lemma}
\begin{proof}
Consider the series 
\begin{equation}\label{59}
f^{*}(x)=\sum_{n=-\infty}^{\infty}c_{n}\varphi(x,\lambda_{n}),
\end{equation}
where
\begin{equation}\label{60}
c_{n}:=\frac{1}{\alpha_{n}}\int_{0}^{\pi}\tilde{\varphi}(x,\lambda_{n})f(x)dx.
\end{equation}
Using Theorem \ref{6} and integrating by parts , we get
\[
c_{n}=\frac{1}{\alpha_{n}\lambda_{n}}\int_{0}^{\pi}\left[-\frac{\partial}{\partial x}
\tilde{\varphi}(x,\lambda_{n})B+\tilde{\varphi}(x,\lambda_{n})\Omega(x)\right]f(x)dx
\]
\[
=-\frac{1}{\alpha_{n}\lambda_{n}}\left[\tilde{\varphi}(\pi,\lambda_{n})Bf(\pi)-
\tilde{\varphi}(0,\lambda_{n})Bf(0)\right]+
\]
\[
+\frac{1}{\alpha_{n}\lambda_{n}}\int_{0}^{\pi}\tilde{\varphi}(x,\lambda_{n})
\left[Bf'(x)+\Omega(x)f(x)\right]dx.
\]
Applying the asymptotic formulas in Theorem \ref{1}, we find $\left\{c_{n}\right\}\in l_{2}.$ Consequently
the series (\ref{59}) converges absolutely and uniformly on $[0,\pi]$. According to (\ref{57}) and (\ref{60}), we have
\[
\int_{0}^{\pi}\tilde{g}(x)f(x)dx=\sum_{n=-\infty}^{\infty}\frac{1}{\alpha_{n}}
\left(\int_{0}^{\pi}\tilde{g}(t)\varphi(t,\lambda_{n})dt\right)\left(\int_{0}^{\pi}
\tilde{\varphi}(t,\lambda_{n})f(t)dt\right)
\]
\[
=\sum_{n=-\infty}^{\infty}c_{n}\left(\int_{0}^{\pi}\tilde{g}(t)\varphi(t,\lambda_{n})dt\right)
=\int_{0}^{\pi}\tilde{g}(t)\left(\sum_{n=-\infty}^{\infty}c_{n}\varphi(t,\lambda_{n})\right)
=\int_{0}^{\pi}\tilde{g}(t)f^{*}(t)dt.
\] Since $g(x)$ is arbitrary, $f(x)=f^{*}(x)$ is obtained, i.e., the expansion formula (\ref{58}) is found.
\end{proof}

\subsection{Derivation of Boundary Condition}

\begin{lemma} \label{lm4}
The following equality holds:
\begin{equation}\label{61}
\sum_{n=-\infty}^{\infty}\frac{\varphi(x,\lambda_{n})}{\alpha_{n}\beta_{n}}=0.
\end{equation}
\end{lemma}
\begin{proof}
Using residue theorem, we get
\begin{equation}\label{62}
\sum_{n=-\infty}^{\infty}\frac{\varphi(x,\lambda_{n})}{\alpha_{n}\beta_{n}}=
\sum_{n=-\infty}^{\infty}\frac{\varphi(x,\lambda_{n})}{\dot{\Delta}(\lambda_{n})}=
\sum_{n=-\infty}^{\infty}\underset{\lambda=\lambda_{n}}{Res}\frac{\varphi(x,\lambda)}{\Delta(\lambda)}
=\frac{1}{2\pi i}\int_{\Gamma_{N}}\frac{\varphi(x,\lambda)}{\Delta(\lambda)}d\lambda,
\end{equation} 
where $\Gamma_{N}=\left\{\lambda:\left|\lambda\right|=N+\frac{1}{2}\right\}$. From (\ref{a}) and (\cite{Mar}, Lemma 3.4.2),
\begin{equation}\label{71}
\Delta(\lambda)=\lambda\sin\lambda\pi+O(e^{\left|Im\lambda\right|\pi}). 
\end{equation}
We denote $G_{\delta}=\left\{\lambda:\left|\lambda-n\right|\geq\delta, \ \ n=0,\pm 1,\pm 2...\right\}$
for some small fixed $\delta>0$ and $\left|\sin\lambda\pi\right|\geq C_{\delta}e^{\left|Im\lambda\right|\pi}$,
$\lambda\in G_{\delta}$, where $C_{\delta}$ positive number. Therefore, we have
\[
\left|\Delta(\lambda)\right|\geq C_{\delta}\left|\lambda\right|e^{\left|Im\lambda\right|\pi}, \ \ \lambda\in G_{\delta}.
\] Using this inequality and (\ref{11}), we obtain (\ref{61}). 
\end{proof}
\begin{theorem} \label{thm7}
The following relation is valid:
\begin{equation}\label{63}
\left(\lambda_{n}+h_{1}\right)\varphi_{1}(\pi,\lambda_{n})+h_{2}\varphi_{2}(\pi,\lambda_{n})=0.
\end{equation}
\end{theorem}

\begin{proof}
From (\ref{61}), we can write for any $n_{0}\in\mathbb{Z}$
\begin{equation}\label{64}
\frac{\varphi(x,\lambda_{n_{0}})}{\alpha_{n_{0}}}=-\sum_{\stackrel{n=-\infty}{n\neq n_{0}}}^{\infty}
\frac{\beta_{n_{0}}\varphi(x,\lambda_{n})}{\alpha_{n}\beta_{n}}
\end{equation}
Let $m\neq n_{0}$ be any fixed number and $f(x)=\varphi(x,\lambda_{k})$. Then substituting (\ref{64}) in (\ref{58})
\[
\varphi(x,\lambda_{k})=\sum_{\stackrel{n=-\infty}{n\neq n_{0}}}^{\infty}c_{nk}\varphi(x,\lambda_{n}),
\] where
\[
c_{nk}=\frac{1}{\alpha_{n}}\int_{0}^{\pi}\left[\tilde{\varphi}(t,\lambda_{n})-\frac{\beta_{n_{0}}}{\beta_{n}}
\tilde{\varphi}(t,\lambda_{n_{0}})\right]\varphi(t,\lambda_{k})dt.
\] The system of functions $\{\varphi_{0}(x,\lambda_{n})\}$, $\left(n\in\mathbb{Z}\right)$ 
is orthogonal in $L_{2}(0,\pi;\mathbb{C}^{2})$.
Then by (\ref{3}), the system of functions $\{\varphi(x,\lambda_{n})\},$ $\left(n\in\mathbb{Z}\right)$ 
is orthogonal in $L_{2}(0,\pi;\mathbb{C}^{2})$ as well. 
Therefore, $c_{nk}=\delta_{nk},$ where $\delta_{nk}$ is Kronecker delta. Let us define 
\begin{equation}\label{65}
a_{nk}:=\int_{0}^{\pi}\tilde{\varphi}(t,\lambda_{n})\varphi(t,\lambda_{k})dt. 
\end{equation} 
Using this expression, we have for $n\neq k$
\begin{equation}\label{66}
a_{kk}-\frac{\beta_{n}}{\beta_{k}}a_{nk}=\alpha_{k}.
\end{equation}
It follows from (\ref{65}) that $a_{nk}=a_{kn}$. Taking into account this equality and (\ref{66}),  
\[
\beta_{k}^{2}\left(\alpha_{k}-a_{kk}\right)=\beta_{n}^{2}\left(\alpha_{n}-a_{nn}\right)=H, \ \ \ k\neq n,
\] where $H$ is a constant.
Then, we have
\[
\int_{0}^{\pi}\tilde{\varphi}(t,\lambda_{n})\varphi(t,\lambda_{n})dt=\alpha_{n}-\frac{H}{\beta_{n}^{2}}
\]
and
\[
\int_{0}^{\pi}\tilde{\varphi}(t,\lambda_{k})\varphi(t,\lambda_{n})dt=-\frac{H}{\beta_{k}\beta_{n}}, \ \ k\neq n.
\] 
It is easily obtained that for $k\neq n$, 
\[
\int_{0}^{\pi}\left[\varphi_{1}(x,\lambda_{k})\varphi_{1}(x,\lambda_{n})+
\varphi_{2}(x,\lambda_{k})\varphi_{2}(x,\lambda_{n})\right]dx=
\]
\[
=\frac{1}{(\lambda_{k}-\lambda_{n})}\left[\varphi_{2}(\pi,\lambda_{k})\varphi_{1}(\pi,\lambda_{n})
-\varphi_{1}(\pi,\lambda_{k})\varphi_{2}(\pi,\lambda_{n})\right]
=-\frac{H}{\beta_{k}\beta_{n}}.
\] According to the last equation, for $n\neq k,$
\begin{equation}\label{67}
\beta_{k}\varphi_{2}(\pi,\lambda_{k})\beta_{n}\varphi_{1}(\pi,\lambda_{n})
-\beta_{k}\varphi_{1}(\pi,\lambda_{k})\beta_{n}\varphi_{2}(\pi,\lambda_{n})=-H{(\lambda_{k}-\lambda_{n})}.
\end{equation}
We denote 
\begin{equation}\label{69}
D_{n}:=\beta_{n}\varphi_{1}(\pi,\lambda_{n}), \ \ \ \  
E_{n}:=\beta_{n}\varphi_{2}(\pi,\lambda_{n}).
\end{equation}
Then, we can rewrite equation (\ref{67}) as follows
\begin{equation}\label{68}
D_{k}E_{n}-E_{k}D_{n}=H{(\lambda_{k}-\lambda_{n})}, \ \ \ n\neq k.
\end{equation} 
Let $i,\ j,\ k, \ n$ be pairwise distinct integers, then we get
\[
\begin{array}{c}
D_{k}E_{n}-E_{k}D_{n}=H{(\lambda_{k}-\lambda_{n})}, \\
D_{n}E_{i}-E_{n}D_{i}=H{(\lambda_{n}-\lambda_{i})}, \\
D_{i}E_{k}-E_{i}D_{k}=H{(\lambda_{i}-\lambda_{k})}.
\end{array}
\] Adding them together, we find
\[
D_{n}(E_{i}-E_{k})+E_{n}(D_{k}-D_{i})=E_{i}D_{k}-D_{i}E_{k}.
\] In this equation, replacing $n$ by $j$, we get another equation
\[
D_{j}(E_{i}-E_{k})+E_{j}(D_{k}-D_{i})=E_{i}D_{k}-D_{i}E_{k}.
\] Subtracting the last two equation,
\[
(D_{n}-D_{j})(E_{i}-E_{k})=(D_{i}-D_{k})(E_{n}-E_{j}).
\] 
In the case of $E_{n}=E_{j}$, for some $n,\ j\in\mathbb{Z}$, then $E_{n}=$const. From (\ref{68}), 
$D_{n}=\mu_{1}\lambda_{n}+\mu_{2}$. In the case of $E_{n}\neq E_{j}$, then
we obtain $D_{n}=\mu_{1}\lambda_{n}+\mu_{2}$ and $E_{n}=\mu_{3}\lambda_{n}+\mu_{4}$, 
where in both cases $\mu_{1},\ \mu_{2},\ \mu_{3},\ \mu_{4}$ are constant.
Therefore, using these relation in (\ref{69}), we find
\[
\beta_{n}\varphi_{1}(\pi,\lambda_{n})=\mu_{1}\lambda_{n}+\mu_{2}, \ \ \
\beta_{n}\varphi_{2}(\pi,\lambda_{n})=\mu_{3}\lambda_{n}+\mu_{4}.
\]
Using 
\[
\varphi_{1}(\pi,\lambda_{n})=O\left(\frac{1}{n}\right), \ \ \ 
\varphi_{2}(\pi,\lambda_{n})=(-1)^{n+1}+O\left(\frac{1}{n}\right), 
\] $\lambda_{n}=n+O\left(\frac{1}{n}\right)$ and $\beta_{n}=n(-1)^{n}+O(1)$ derived from (\ref{7}) and (\ref{71}), 
we obtain $\mu_{1}=0$, $\mu_{3}=-1$. Denoting $h_{2}:=\mu_{2}$ and $h_{1}:=-\mu_{4}$, 
\[
h_{2}\varphi_{2}(\pi,\lambda_{n})=-\left(\lambda_{n}+h_{1}\right)\varphi_{1}(\pi,\lambda_{n}), \ \ n\in\mathbb{Z}
\] 
is obtained and it follows from (\ref{68}) that $H=h_{2}$.
\end{proof}

Thus, we have proved the following theorem:

\begin{theorem}
For the sequences $\left\{\lambda_{n}, \alpha_{n}\right\}$, $\left(n\in\mathbb{Z}\right)$, to be the spectral data for a
certain boundary value problem $L(\Omega(x), h_{1}, h_{2})$ of the form (\ref{1}), (\ref{2}) with $\Omega(x)\in L_{2}(0,\pi)$
it is necessary and sufficient that the relations (\ref{8}) and (\ref{10}) hold.
\end{theorem}

Algorithm of the construction of the function $\Omega(x)$ by spectral data $\left\{\lambda_{n},\alpha_{n}\right\}$,
$\left(n\in\mathbb{Z}\right)$ follows from the proof of the theorem: \\
1) By the given numbers $\left\{\lambda_{n},\alpha_{n}\right\}$,
$\left(n\in\mathbb{Z}\right)$ the function $F(x,t)$ is constructed by formula (\ref{29}), \\
2) The function $K(x,t)$ is found from equation (\ref{28}), \\
3) $\Omega(x)$ calculated by the formula (\ref{4}).

\begin{acknowledgement}
This work is supported by The Scientific and Technological Research Council
of Turkey (T\"{U}B\.{I}TAK).
\end{acknowledgement}


\end{document}